\documentclass{cimart}

\title [A note on division rings satisfying generalized rational identities with anti-automorphisms]{A note on division rings satisfying generalized\\ rational identities with anti-automorphisms}

\author{% Please, use "Firstname Lastname" format, without abreviations
    Vo Hoang Minh Thu, Vu Mai Trang
    }

\authorinfo[% Format: please, use "F. Name"
    Vo Hoang Minh Thu]{% Format: please, use "Department (only if 
    Faculty of Mathematics and Computer Science, University of Science, Ho Chi Minh City, Vietnam; Vietnam National University, Ho Chi Minh City, Vietnam}{%
    minhthuvohoang@gmail.com
    }

\authorinfo[% Format: please, use "S. Name"
    Vu Mai Trang]{% Format: please, use "Department (only if 
    Binh Duong University, Ho Chi Minh City, Vietnam}{%
    vmtrang@bdu.edu.vn
    }

\abstract{%
    	Let $D$ be a division ring with infinite center $F$, $\sigma$ be an  anti-automorphism of $D$ and $m$ be a positive integer such that  $\sigma^m\neq \mathrm{Id}$. In this paper, we show that if $D$ satisfies a $\sigma^m$-GRI, then $D$ is centrally finite. 
    }



\VOLUME{34}
\YEAR{2026}
\ISSUE{1}
\NUMBER{3}
\DOI{https://doi.org/10.46298/cm.15434}

\begin{document}

% Here is where the main text should be typed:
	\section{Introduction and the main result} 

Let $D$ be a division ring with center $F$. We say that $D$ is \textit{centrally finite} if $D$ is regarded as a finite-dimensional vector space over $F$. This topic relates to conditions under which a division ring is centrally finite. We focus on division rings satisfying certain identities.

A classical result states that if $D$ satisfies a polynomial identity, then $D$ is centrally finite (see \cite{i}). To present this result and the extended results we aim to prove, we recall some concepts. Let $X=\{x_1,x_2,\dots, x_n\}$ be a set of $n$ non-commutative indeterminates and $K$ be a field. We denote $K\langle X\rangle$ as the polynomial ring in indeterminates $X$ over $K$, and $R$ as a ring whose center contains $K$. A polynomial $f(X)\in K\langle X\rangle \backslash \{0\}$ is called a \textit{polynomial identity} of $R$ if for every $n$-tuple $(a_1,a_2,\dots, a_n)\in R^n$, we have $f(a_1,a_2,\dots,a_n)=0$. In this case, we say $R$ satisfies the \textit{polynomial identity} $f$, abbreviated as PI. If $R$ satisfies some polynomial identity, then we say $R$ satisfies a polynomial identity. Since we only work with division rings, we directly assume $R=D$ is a division ring and $K$ coincides with the center $F$ of $D$.

The most classical result concerning division rings satisfying polynomial identities relates to its finite-dimensionality. Specifically, it is proven in \cite{i} that if $D$ is a division ring with center $F$ and satisfies some polynomial identity over $F$, then $D$ is centrally finite. This result has been subsequently extended in many different directions (e.g., see \cite{s, s1, k, ber1, ber2, chiba, khar, w, w2, j}). 

We now recall the basic notation used throughout. We denote $D\langle X\rangle=D*_F F\langle X\rangle$ as the free product of $D$ with $F\langle X\rangle$ over $F$. Each element in $D\langle X\rangle\backslash \{0\}$ is called a \textit{generalized polynomial} over $D$. Furthermore, since the ring of generalized polynomials $D\langle X\rangle$ is fir, the maximal right ring of quotients of $D\langle X\rangle$ is a division ring (e.g., see \cite{p}). This division ring is denoted as $D(X)$. Each nonzero element in $D(X)$ is called a \textit{generalized rational polynomial} over $D$. Suppose $f(X)$ is a generalized rational polynomial (respectively, a generalized polynomial) over $D$. If for every $n$-tuple $(a_1,a_2,\dots,a_n)\in D^n$ such that when $f(a_1,a_2,\dots,a_n)$ is defined, one has $f(a_1,a_2,\dots,a_n)=0$, then we say $D$ satisfies the \textit{generalized rational identity}, abbreviated as GRI (respectively, the \textit{generalized polynomial identity}, abbreviated GPI) $f$, or $f$ is a \textit{generalized rational  identity} (respectively, a \textit{generalized polynomial identity}) of $D$.

In 1966, Amitsur proved that if the center $F$ is infinite and the division ring $D$ satisfies a generalized rational identity, then $D$ is centrally finite (see \cite{s1}). There are several extensions of this result. For example, in 1996, Chiba proved that if the center $F$ is infinite and the multiplicative group $D^*=D\backslash \{0\}$ contains a non-central proper subnormal subgroup satisfying a generalized rational identity, then $D$ is centrally finite (see \cite{chiba_96}). Another direction is to consider identities corresponding to an anti-automorphism. Recall that a bijection $\sigma : D\to D, x\mapsto x^\sigma,$ is called an \textit{anti-automorphism} of $D$ if $(a+b)^\sigma=a^\sigma+b^\sigma$ and $(ab)^\sigma=b^\sigma a^\sigma$ for all $a,b\in D$. In the case where the order of $\sigma$ is $2$, $\sigma$ is called an \textit{involution} of $D$. Let $m$ be a positive integer such that $\sigma^i\neq \mathrm{Id}$ for all $1\leq i\leq m$. For each $1\leq i\leq m$, denote $X^{\sigma^i}=\{x_1^{\sigma^i},x_2^{\sigma^i},\dots\}$ as the set of indeterminates indexed corresponding to each $\sigma^i$ and $X_m=\cup_{i=0}^mX^{\sigma^i}$. We denote $D(X_m)$ as the division ring as mentioned above. Let $f(X_m)$ be a nonzero element in $D(X_m)$. If for every $n$-tuple $(r_1,r_2,\dots,r_n)\in D^n$ such that when $$f(r_1,r_2,\dots,r_n,r_1^\sigma,r_2^\sigma,\dots,r_n^\sigma,\dots,r_1^{\sigma^m},r_2^{\sigma^m},\dots,r_n^{\sigma^m})$$ is defined, we have $$f(r_1,r_2,\dots,r_n,r_1^\sigma,r_2^\sigma,\dots,r_n^\sigma,\dots,r_1^{\sigma^m},r_2^{\sigma^m},\dots,r_n^{\sigma^m})=0,$$ then we say $f(X_m)$ is a \textit{generalized rational identity} of $D$ \textit{with respect to} $\sigma^m$. In this case, we simply say that $D$ satisfies a $\sigma^m$-GRI $f$ or $f$ is a $\sigma^m$-GRI of $D$. The ``best'' result is related to when $\sigma$ is an involution. In this case, $m=1$ and $\sigma$ will be denoted as $\star$. In more detail, it is proved that if the center $F$ is infinite and $D$ satisfies a $\star$-GRI, then $D$ is centrally finite (see \cite{Pa_Des}). The case of generalized polynomial identities with respect to $\sigma$ receives even more attention (e.g., see the book \cite{k} for a survey and the paper \cite{swain} for some additional interesting results on this topic). The objective of this paper is to prove the following result.

\begin{theorem}\label{main}
	Let $D$ be a division ring with  infinite center $F$. If $D$ satisfies a $\sigma^m$-GRI then $D$ is centrally finite. 
\end{theorem}

This theorem can be viewed as an anti-automorphism counterpart of Amitsur's Theorem and a natural extension of \cite{Pa_Des}. 

\section{Preliminaries} 

In short and simple terms, we prove the main result of this paper by induction on $n$, the number of indeterminates in the identity. In this section, we prove the case for $n=1$. To do it, we need to review some basic concepts and establish notation. Many results in this section may have proofs similar to those found in some reference documents, but none of the results are directly quoted from them. Therefore, for the convenience of the readers, we will provide proofs for the results if necessary.  The main result we obtained in this section is Lemma ~\ref{2.4}, which states that if a division ring satisfies a blended linear identity in one variable, then this division ring is centrally finite.

We emphasize some notations that will be used throughout the paper as follows. Let $D$ be a division ring with center $F$, and $\sigma: D\to D, x\mapsto x^\sigma,$ be an anti-automorphism over $F$, that is, $(a+b)^\sigma=a^\sigma+b^\sigma, (ab)^\sigma=b^\sigma a^\sigma$ and $\alpha^\sigma=\alpha$, for every $a,b\in D$ and $\alpha\in F$. Consider $X=\{x_1,x_2,\dots\}$ as a countable set of non-commutative indeterminates. Suppose $m$ is an integer such that $\sigma^i\neq Id$ for all $1\leq i\leq m$. For each $1\leq i\leq m$, let $X^{\sigma^i}=\{x_1^{\sigma^i},x_2^{\sigma^i},\dots\}$ and $X_m=\bigcup_{i=0}^mX^{\sigma^i}$. The notation $F\langle X_m\rangle$ denotes the free algebra generated by $X_m$ over $F$, and $D\langle X_m\rangle=D*_FF\langle X_m\rangle$ is the free product of $D$ and $F\langle X_m\rangle$ over $F$. Each element of $D\langle X_m\rangle$ has the form 
\[
f=f(x_1,x_2,\dots,x_n,x_1^\sigma,x_2^\sigma,\dots,x_n^\sigma,\dots,x_1^{\sigma^m},x_2^{\sigma^m},\dots,x_n^{\sigma^m})=\sum_{i=1}^lP_i,
\]
where $P_i=\delta_ia_{i_0}x_{i_1}^{\sigma^{k_1}}a_{i_1}x_{i_2}^{\sigma^{k_2}}a_{i_2}\cdots x_{i_q}^{\sigma^{k_q}}a_{i_q}$ are monomials, $\delta_i$ is in $F$, $\{a_{i_0},a_{i_1}, \dots, a_{i_q}\}$ is a subset of $ D$, and $\{x_{i_1}^{\sigma^{k_1}},x_{i_2}^{\sigma^{k_2}},\dots,x_{i_q}^{\sigma^{k_q}}\}$ is a subset $X_m$.

For each monomial $P\in D\langle X_m\rangle$, we define the \textit{$\sigma^m$-degree of $P$ with respect to the indeterminate $x_j$}, denoted $\sigma^m$-$\deg_{x_j}(P)$, as the number of occurrences of $x_j^{\sigma^i}$ in $P$ for some $0\leq i\leq m$. Then, the degree of $P$, denoted by $\deg(P)$, is defined as
\[
\deg(P)=\sum_{x_j\in X}\sigma^m\text{-}\deg_{x_j}(P).
\]
We also define the \textit{$\sigma^m$-height of $P$ with respect to the indeterminate $x_j$}, denoted $\sigma^m\text-ht_{x_j}(P)$, as the maximum value between $(\sigma^m\text{-}\deg_{x_j}(P)-1)$ and $0$. Similarly, the height of $P$, denoted by $ht(P)$, is defined as
\[
ht(P)=\sum_{x_j\in X}\sigma^m\text{-}ht_{x_j}(P).
\]
In other words, the height of $P$ is equal to the degree of $P$ minus the number of distinct indeterminates appearing in $P$. Let $\AA$ be a $F$-basis of $D$ containing $1$.  Then $\AA$, combined with the standard $F$-basis of $F\langle X_m\rangle$, induces a monomial basis $M(\AA)$ for $D\langle X_m\rangle$. Suppose $f=\sum_{i=1}^\ell P_i\in D\langle X_m\rangle$ where $P_1, P_2,\dots,P_\ell$ are monomials in $M(\AA)$. We also define
\begin{enumerate}
	\item $\sigma^m\text{-}\deg_{x_j}(f)=\max\{\sigma^m\text{-}\deg_{x_j}(P_i)\mid 1\leq i\leq l\}$;
	\item $\deg(f)=\max\{\deg(P_i)\mid 1\leq i\leq l\}$;
	\item $\sigma^m\text{-}ht_{x_j}(f)=\max\{\sigma^m\text{-}ht_{x_j}(P_i)\mid 1\leq i\leq l\}$;
	\item $ht(f)=\max\{ht(P_i)\mid 1\leq i\leq l\}$.
\end{enumerate}

Assume that $x_1,x_2,\dots,x_n$ are indeterminates appearing in the polynomial $f$. Suppose indeterminate $x_j$ appears in all monomials of $f$, then we say $f$ is \textit{blended} with respect to the indeterminate $x_j$. If $f$ is blended with respect to every indeterminate $x_j$ where $1\leq j\leq n$, then we say $f$ is \textit{blended}.

Suppose 
\[
f=f(x_1,x_2,\dots,x_n,x_1^\sigma,x_2^\sigma,\dots,x_n^\sigma,\dots,x_1^{\sigma^m},x_2^{\sigma^m},\dots,x_n^{\sigma^m})
\]
is a $\sigma^m$-GPI of $D$ with $\deg(f)=p$. Let $q$ be the number of indeterminates in $f$ that are not blended. Suppose $q>0$. Without loss of generality, we assume $f$ is not blended with respect to indeterminates $x_1,x_2,\cdots,x_q$. Set
\begin{align*}
f_1 &= f(0,x_2,\dots,x_n,0,x_2^\sigma,\dots,x_n^\sigma,\dots,0,x_2^{\sigma^m},\dots,x_n^{\sigma^m}); \\
f_2 &= f-f_1.
\end{align*}
Since $f$ is a $\sigma^m$-GPI of $D$, one has $f_1$ and $f_2$ are also $\sigma^m$-GPIs of $D$. It is clear that $\deg(f_1)$ and $\deg(f_2)$ are less than or equal to $p$, and $f_2$ is blended with respect to indeterminate $x_1$. Similarly, for $2\leq j\leq q$, set
\begin{align*}
f_{2j-1} = f_{2j-2}(x_1,x_2,\dots,x_{j-1},0,
x_{j+1},\cdots,x_n,\cdots,x_1^{\sigma^m},x_2^{\sigma^m},\dots,x_{j-1}^{\sigma^m},0,x_{j+1}^{\sigma^m},\cdots,x_n^{\sigma^m}),
\end{align*}
and $f_{2j}=f_{2j-2}-f_{2j-1}$. Then, the polynomial $f_{2q}$ is blended and is also a $\sigma^m$-GPI of $D$. From the above reasoning, we obtain the following proposition.

\begin{proposition} 
	Let $D$ be a division ring and $\sigma$ be an anti-automorphism of $D$. If $D$ satisfies a $\sigma^m$-GPI, then $D$ satisfies a $\sigma^m$-GPI that is blended.	
\end{proposition}

We continue to consider $f = \sum_{i=1}^\ell P_i \in D\langle X_m \rangle$ where $P_1,P_2,\dots,P_\ell$ are monomials. Then, $f$ is called \textit{$\sigma^m$-multilinear} (or shortly, \textit{$\sigma^m$-linear}) of degree $n$ on indeterminates $x_1,x_2,\dots,x_n$ if each monomial $P_i$ in $f$ satisfies the following properties:
\begin{enumerate}
	\item $\deg(P_i) = n$;
	\item $\sigma^m\text{-}\deg_{x_j}(P_i) = 1$ for all $x_j \in \{x_1,x_2,\dots,x_n\}$.
\end{enumerate}
In other words, the polynomial $f$ is $\sigma^m$-linear if $f$ is a polynomial blended of height $0$.

Suppose $f$ is a blended $\sigma^m$-GPI of $D$ with $\{x_1,x_2,\dots,x_n\}$ being the set of all indeterminates appearing in $f$. Let $p = \deg(f)$ and $p_j = \sigma^m\text{-}\deg_{x_j}(f)$ $(1 \leq j \leq k)$. We write $f = f(x_j)$ to emphasize considering $f$ as a polynomial in the indeterminate $x_j$. If $ht(f) = t > 0$, we choose $x_{j_q} \in X$ such that $x_{j_q}$ does not appear in $f$ and set
\[
\Delta_jf = f(x_j + x_{j_q}) - f(x_j) - f(x_{j_q}).
\]
Then,
\begin{enumerate}
	\item $\deg(\Delta_jf) = \deg(f)$;
	\item $\sigma^m\text{-}\deg_{x_j}(\Delta_jf) = \sigma^m\text{-}\deg_{x_{j_q}}(\Delta_jf) = p_j - 1$;
	\item $\sigma^m\text{-}\deg_{x_{j_k}}(\Delta_jf) \leq \sigma^m\text{-}\deg_{x_{j_k}}(f)$ for $x_{j_k} \neq x_j,x_{j_q}$.
\end{enumerate}
Therefore, $\Delta_jf$ is also a blended $\sigma^m$-GPI of $D$ satisfying $ht(\Delta_jf) \leq t - 1$. Define
\[
\begin{aligned}
\Delta_j^{(0)}f &= f, \\
\Delta_j^{(1)}f &= \Delta_jf, \\
\Delta_j^{(v)}f &= \Delta_j\Delta_j^{(v-1)}f \quad (v \in \mathbb{N}, v > 1).
\end{aligned}
\]
Then, if $\sigma^m\text{-}\deg_{x_j}(f) = p_j$, we have $\sigma^m\text{-}\deg_{x_j}(\Delta_j^{(p_j-1)}f) = 1$. From here, we have the following lemma.

\begin{lemma} Given $f = \sum_{i=1}^\ell P_i \in D\langle X_m \rangle$ is a blended polynomial in indeterminates $x_1, x_2, \dots, x_n$ and has height $t$. Let $P$ be a monomial in $f$ such that $ht(P) = t$. Then, we have that $\Delta_1^{(p_1-1)}\cdots\Delta_n^{(p_n-1)}f$ is a $\sigma^m$-linear polynomial in $\sum_{j=1}^n p_j$ indeterminates, where $p_j = \sigma^m\text{-}\deg_{x_j}(P)$ $(1 \leq j \leq n)$.
\end{lemma}
\begin{proof}
	If $f$ has height $t=0$, then $p_j=1$ and $\Delta_jf=f$ is $\sigma^m$-linear for every $1\leq j\leq n$, thus the lemma is proved. Consider the case $t>0$ and $P$ is a monomial in $f$ such that $\sigma^m\text{-ht}(P)=t$. With $p_j=\sigma^m\text{-}\deg_{x_j}(P)$, we have
	\[
	\Delta_1^{(p_1-1)}\cdots\Delta_n^{(p_n-1)}f=\sum\Delta_1^{(p_1-1)}\cdots\Delta_n^{(p_n-1)}P_i.
	\]
	Choose $j_0$ such that $\sigma^m\text{-}\deg_{x_{j_0}}(P)=p_{j_0}>1$, $\deg_{x_i}(P)=p_i=1$ and $1\leq j_0<i\leq n$. Then
	\[
	\Delta_1^{(p_1-1)}\cdots\Delta_n^{(p_n-1)}P=\Delta_1^{(p_1-1)}\cdots\Delta_{j_0}^{(p_{j_0}-1)}P.
	\]
	Let $x_{n+1}\in X$, set
	\[
	\Delta_{j_0}P=P(x_{j_0}+x_{n+1})-P(x_{j_0})-P(x_{n+1})=\sum_{u=1}^{p_{j_0}-1}h_u
	\]
	in which $h_u$ is a polynomial such that 			$$\sigma^m\text{-}\deg_{x_{n+1}}(h_u)=u\text{ and } \sigma^m\text{-}\deg_{x_{j_0}}(h_u)=p_{j_{0}}-u.$$
	Each polynomial $h_u$ is a sum of $C_{p_{j_0}}^u$ monomials $P_{u_k}$, where each monomial $P_{u_k}$ maps to $P$ under the projection $x_{n+1}\mapsto x_{j_0}$ and $\sigma^m\text{-}\deg_{x_{j_0}}(P_{u_k})=p_{j_0}-u$.
	Hence,
	\[
	\Delta_{j_0}^{(p_{j_0}-1)}P=\Delta_{j_0}^{(p_{j_0}-2)}\Delta_{j_0}^{(1)}f=\Delta_{j_0}^{(p_{j_0}-2)}\sum_{u=1}^{p_{j_0}-1}h_u=\sum_{u=1}^{p_{j_0}-1}\Delta_{j_0}^{(p_{j_0}-2)}h_u.
	\]
	For $u>1$, since $\sigma^m$-$\deg_{x_{j_0}} (h_u) =p_{j_0}-u<p_{j_0}-1$, we have $\Delta_{{{j_0}}}^{(p_{j_0}-2)}(h_u)=0$. Therefore,
	\[
	\Delta_{{j_0}}^{p_{j_0}-1} P=\Delta^{(p_{j_0}-2)}_{{j_0}}h_1.
	\]
	By induction, we obtain that  $\Delta_{j}^{p_j-1}P$ has
	\[
	C^1_{p_j}C^1_{p_j-1}\cdots C^1_2= p_j(p_j-1)\cdots 2=p_j!
	\]
	monomials, each of which maps to $P$ under the projection $x_{e}\mapsto x_{j_0}$ (and here we denote $e=n+1, n+2, \dots, n+(p_{j_0}-1$)). Furthermore, we have 
	\[
	\sigma^m\text{-}\deg_{x_{j_0}}( \Delta_{j_0}^{(p_{j_0}-1)}P)=\sigma^m\text{-}\deg_{x_{e}} (\Delta_{x_{e}}^{(p_{j_0}-1)}P)=1
	\]
	for $e\in\{n+1, n+2, \cdots, n+(p_{j_0}-1)\}$. 
	
	The remaining part of the lemma can be established by induction on $n$.  Finally, we obtain $\Delta^{(p_1-1)}_{{1}}\cdots \Delta^{(p_n-1)}_{n} P$ as a linear polynomial in $\sum^{n}_{j=1} p_j$ indeterminates. Moreover, these monomials all map to $P$ under the projection satisfying the property: for every $1\leq j\leq n$, the projection of $x_{e}$ to $x_{j}$ is such that $$  j+\sum_{u=j+1}^{n} p_u\leq  e\leq j+\sum_{u=j}^{n} p_u.$$
	Let $ P_{i_0}\ne P$ be a monomial of $f$. If  there exists  $1\leq j_k\leq n$ satisfying the inequality  $\sigma $-$\deg_{x_{{j_k}}}(P_{i_0})< p_{j_k}=\sigma$-$\deg_{x_{j_k}}(P) $, then $\Delta_{{j_k}}^{(p_{j_k}-1)}P_{i_0}=0 $. So 
	\begin{multline*}
	\Delta^{(p_1-1)}_{{1}}\cdots\Delta^{(p_{j_k-1}-1)}_{j_k-1}\Delta^{(p_{j_k}-1)}_{j_k+1}\Delta^{(p_{j_k+1}-1)}_{j_k+1}\cdots \Delta^{(p_n-1)}_{n} P_{i_0} \\ 
	= \Delta^{(p_1-1)}_{{1}}\cdots\Delta^{(p_{j_k-1}-1)}_{j_k-1}\Delta^{(p_{j_k+1}-1)}_{j_k+1}\cdots \Delta^{(p_l-1)}_{x_{k_l}}  \left(\Delta^{(p_{j_k}-1)}_{x_{k_{j_k}}} P_{i_0}\right) =0.
	\end{multline*}    
	If $\Delta^{(p_1-1)}_{1}\cdots \Delta^{(p_n-1)}_{{n}} P_{i_0}\ne 0$ then $p_j^{'}=\sigma^m$-$\deg_{x_j}(P_{i_0})\geq p_j ,$ with $1\leq j\leq n. $ However, observe that $\sum (p_j-1)=t\geq \sum (p^{'}_j-1)$, which implies that $ p^{'}_{j}=p_j$ for $1\leq j\leq n.$ Using a similar method as for $P$, we have that $\Delta^{(p_1-1)}_{1}\cdots \Delta^{(p_n-1)}_{{n}} P_{i_0}$ is a multilinear polynomial in $\sum p_j$  indeterminates just like $\Delta^{(p_1-1)}_{1}\cdots \Delta^{(p_n-1)}_{n} P$. Hence, $\Delta^{(p_1-1)}_{1}\cdots \Delta^{(p_n-1)}_{{n}} f$ is a blended $\sigma^m$-linear polynomial.
\end{proof}

From the lemma above, we immediately obtain the following corollary.
\begin{corollary} \label{2.3}
	If the division ring $D$ satisfies an $\sigma^m$-GPI of degree $n$, then there exists an $\sigma^m$-GPI of $D$ which is blended $\sigma^m$-linear of degree not exceeding $n$.
\end{corollary}

Now, we show the main result of this section by considering the case where $X=\{x\}$. In this case, if $f$ is a linear $\sigma^m$-form in $D\langle X_m\rangle$, then $f$ can be represented as $f=\sum_{i=0}^k g_i(x^{\sigma^i})$, where $g_i(x^{\sigma^i})=\sum_{j=1}^{n_i} a_{ij}x^{\sigma^i}b_{ij}$ and $0\leq k\leq m$.

\begin{lemma}\label{2.4}
	Let $D$ be a division ring with an infinite center and $X=\{x\}$. Suppose $f\neq 0$ is a $\sigma^m$-generalized polynomial in $D\langle X_m\rangle$. If $D$ is not centrally finite, then there exists an element $r\in D$ such that $f(r)\neq 0$.
\end{lemma}
\begin{proof} By Corollary~\ref{2.3}, we assume that $f$ is blended $\sigma^m$-linear. By the hypothesis, $f$ can be represented as
	$$f=\sum_{i=0}^kg_i(x^{\sigma^i})$$
	where $g_i(x^{\sigma^i})=\sum_{j=1}^{n_i}a_{ij}x^{\sigma^i}b_{ij}$.
	We can choose the largest non-negative integer $0\leq k\leq m$ such that $g_k(x^{\sigma^k})\neq 0$. Consider the case $k=0$, then
	$$f=f(x)=g_0(x)=\sum_{j=1}^{n_0}a_{0j}xb_{0j}.$$
	Suppose $\sum_{j=1}^{n_0}a_{0j}rb_{0j}=0$ for all $r\in D$. With the assumption at hand, we deduce that $f(r)=\sum_{j=1}^{n_0}a_{0j}rb_{0j}\in DxD$ is a GPI of $D$. According to \cite[Corollary 6.1.3]{k}, then $\sum_{j=1}^{n_0}a_{0j}xb_{0j}= 0$, which contradicts the assumption. Hence, there exists an element $r\in D$ such that
	$$f(r)=\sum_{j=1}^{n_0}a_{0j}rb_{0j}\neq 0.$$
	Suppose the lemma holds for all $k\leq m-1$. For $k=m$, we have
	$$f =\sum _{i=0}^{m}g_i(x^{\sigma ^{i}})\in D\left\langle X_m\right\rangle\backslash\{0\}.$$
	Suppose conversely that we have
	$\sum _{i=0}^{m}g_i(r^{\sigma ^{i}})= 0,$ for all $r\in D$. Consider the polynomial $g_0(x)=\sum^{n_0}_{j=1}  a_{0j}xb_{0j}$ in two cases $n_{0}=0$ and $n_{0}>0$. First, we consider the case $n_0>0$. If the set $\{a_{01}, a_{02}, \dots, a_{0n_0}\}$ of $n_0$ elements is linearly independent over $F$, then the set of $2 n_0$ polynomials
	\[
    a_{01}, a_{02}, a_{03}, \dots, a_{0n_0}, a_{mn_n}xa_{01}, a_{mn_m}xa_{02}, a_{mn_m}xa_{03},\dots,  a_{mn_m}xa_{0n_0}
    \]
	is also linearly independent over $F$.
	Indeed, consider
	\[
    \alpha_{1} a_{01}+\cdots +\alpha_{n}a_{0n_0}+\beta_1 a_{mn_n}xa_{01}+\cdots+\beta_n a_{mn_m}xa_{0n_0} =0,
    \]
	where $\alpha_j, \beta_j\in F$ and $j\in\{1, 2, \dots,n_0\}$. Then, we have
	\[
    \begin{cases}\alpha_1a_{01}+\alpha_2 a_{02}+\cdots +\alpha_{n}a_{0n_0}&=0 \\ \beta_1 a_{mn_n}xa_{01}+\beta_2 a_{mn_n}xa_{02}+\cdots+	\beta_n a_{mn_m}xa_{0n_0} &= 0.\end{cases}
    \]
	Since $\{a_{01}, a_{02}, \dots, a_{0n_0}\}$ is linearly independent over $F$, we have the following relation: $\alpha_1=\alpha_2=\cdots =\ \alpha_n=0.$ Furthermore, as $\{a_{01}, a_{02}, \dots, a_{0n_0}\}$ is linearly independent over $F$, the set 
    \[
    \{  a_{mn_n}\otimes_F a_{01}, a_{mn_m}\otimes_F a_{02}, \dots,  a_{mn_m}\otimes_F a_{0n_0}\}
    \]
    is also linearly independent over $F$. According to \cite[Remark 6.1.1]{k}, the set 
    \[
    \{a_{mn_n}xa_{01}, a_{mn_m}xa_{02}, \dots,  a_{mn_m}xa_{0n_0}\}
    \]
    is linearly independent over $F$, which implies $\beta _1=\beta_2=\cdots =\beta_n=0$. Thus, there exists an element $r\in D$ such that
	\[
    a_{01}, a_{02}, a_{03}, \dots, a_{0n_0}, a_{mn_n}ra_{01}, a_{mn_m}ra_{02}, a_{mn_m}ra_{03},\dots,  a_{mn_m}ra_{0n_0} 
    \]
	are linearly independent over $F$. Suppose these elements are linearly dependent over $F$ for all $r\in D$, meaning that there exist $\gamma_1,\gamma_2,\dots,\gamma_{2n}$ not all zero such that
	$$\gamma_1a_{01}+\gamma_2a_{02}+\dots+\gamma_{2n}a_{mn}ra_{0n_0}=0$$
	for all $r\in D$, implying that $D$ satisfies a GPI. According to \cite{s1}, $D$ has a finite center, which contradicts the assumption.
	Let 
	$$g=\sum _{i=0}^{m} g_i(\sigma^i(x\sigma^{-m}(a_{mn_m})\sigma^{-m}(r))-a_{mn_m}r\sum _{i=0}^m g_i( x^{\sigma^i}).$$
	Then,
	\begin{align*}
	g&= \sum_{j=1}^{n_0} a_{0j}x\sigma^{-m}(a_{mn_m})\sigma^{-m}(r)b_{0j}- \sum_{j=1}^{n_0} a_{mn_m}ra_{0j}x b_{0j}\\
	&+\sum_{j=1}^{n_1}a_{1j}\sigma^{1-m}(r)\sigma^{1-m}(a_{mn_m})x^{\sigma }b_{1j}-\sum _{j=1}^{n_1}a_{mn_m} r a_{1j}x b_{1j}\\
	&+ \cdots +\sum_{i=1}^{n_m-1}(a_{mj}r a_{mn_m}-a_{mn_m}r a_{mi})x^{\sigma^m}b_{mi}.
	\end{align*}
	Since $$a_{01}, a_{02}, a_{03}, \dots, a_{0n_0}, a_{mn_m}ra_{01}, a_{mn_m}ra_{02}, a_{mn_m}ra_{03},\dots,  a_{mn_m}ra_{0n_0}$$ are linearly independent over $F$, $$a_{01}\otimes_F\sigma^{-m}(a_{mn_m})\sigma^{-m}(r)b_{01},\dots,a_{0n_0}\otimes_F\sigma^{-m}(a_{mn_m})\sigma^{-m}(r)b_{0n_0},$$ $$a_{mn_m}ra_{01}\otimes_F b_{01},\dots, a_{mn_m}ra_{0n_0}\otimes_F b_{0n_0}$$ are also linearly independent over $F$. According to \cite[Remark 6.1.1]{k}, there exists an isomorphism $\phi: D\otimes_FD\to DxD$ such that $a\otimes_Fb \mapsto axb$. Since 
	$$ \sum_{j=1}^{n_0} a_{0j}x\sigma^{-m}(a_{mn_m})\sigma^{-m}(r)b_{0j}- \sum_{j=1}^{n_0} a_{mn_m}ra_{0j}x b_{0j}$$$$= \phi\big(\sum_{j=1}^{n_0} a_{0j}\otimes\sigma^{-m}(a_{mn_m})\sigma^{-m}(r)b_{0j}- \sum_{j=1}^{n_0} a_{mn_m}ra_{0j}\otimes b_{0j}\big),$$ we deduce that $\phi$ is an isomorphism, and $$\sum_{j=1}^{n_0} a_{0j}\otimes\sigma^{-m}(a_{mn_m})\sigma^{-m}(r)b_{0j}- \sum_{j=1}^{n_0} a_{mn_m}ra_{0j}\otimes b_{0j}\neq 0,$$ so the polynomial
	$$ \sum_{j=1}^{n_0} a_{0j}x\sigma^{-m}(a_{mn_m})\sigma^{-m}(r)b_{0j}- \sum_{j=1}^{n_0} a_{mn_m}ra_{0j}x b_{0j}\not\equiv 0.$$  
	Therefore, $ g\not\equiv 0$. Furthermore, for all $r\in D$, $g(r)=0$. Thus, the polynomial $g $ is a $\sigma^m$-GPI of $D$. We can choose $n_m=1$. Then, $g$ is a polynomial not containing $x^{\sigma^m}$. Thus, $D$ satisfies a $\sigma^m$-GPI not containing $x^{\sigma^m}$. According to Corollary \ref{2.3}, $D$ satisfies a linear mixed $\sigma^m$-GPI of the form
	$$h=\sum_{i=0}^{k} h_i(x^{\sigma^i}), $$
	where $h_i(x^{\sigma^i})=\sum_{j=1}^{n_k} c_{ij}x^{\sigma^i} d_{ij}$ and $k\leq (m-1)$. This contradicts the initial assumption. Now, consider the case $n_0=0$. Then,
	\begin{align*}
	f=f( x^{\sigma}, \dots, x^{\sigma^m} )&=\sum_{i=0}^{m}g_i(x^{\sigma^i})\\
	&= \sum_{j=1}^{n_1} a_{1j}x^{\sigma}b_{1j}+ \sum_{j=0}^{n_2}a_{2j}x^{\sigma^2}b_{2j}+\cdots +\sum_{j=0}^{n_m}a_{mj}x^{\sigma^m}b_{mj}.
	\end{align*}
	Since $\sigma $ is an {anti-automorphism} of $D$, for $t\in D$ there exists $r\in D$ such that $\sigma(r)=t$. Therefore, for every $t\in D $, 
	\begin{align*}
	0&=f(   r^{\sigma}, \dots, r^{\sigma^{m}})\\
	&=\sum_{j=1}^{n_1} a_{1j}tb_{1j}+ \sum_{j=0}^{n_2}a_{2j}t^{\sigma}b_{2j}b_{2j}+\cdots +\sum_{j=0}^{n_m}a_{mj}t^{\sigma^{m-1}}b_{mj}.
	\end{align*}
	Thus, $D$ satisfies a $\sigma^m$-GPI of the form 
	$$\sum_{j=1}^{n_1} a_{1j}xb_{1j}+ \sum_{j=0}^{n_2}a_{2j}x^{\sigma}b_{2j}b_{2j}+\cdots +\sum_{j=0}^{n_m}a_{mj}x^{\sigma^{m-1}}b_{mj},	$$ 
	which contradicts the assumption. Therefore, there exists an element $r\in D$ such that $f(r)\neq 0$.
\end{proof}
\section{Proof of the main result}  
In this section, we will present the proof of Theorem~\ref{main}. We start with a very classical result. In fact, in the previous section, we occasionally cited these results directly, but as mentioned, the previous section is supplementary. Therefore, for the sake of clarity, we restate these results here.
\begin{lemma}\label{l3.1} Let $D$ be a division ring with center $F$. If $F$ is infinite, then the following assertions are equivalent. 
	\begin{enumerate}
		\item $D$ is centrally finite.
		\item $D$ satisfies a PI.
		\item $D$ satisfies a GPI.
		\item $D$ satisfies a GRI.
	\end{enumerate} 
\end{lemma}
\begin{proof} For an original proof of the equivalence between (1) and (2), we refer to \cite[Theorem 1]{i}. In addition, it is clear that every PI is, in particular, a GPI, and every GPI is a GRI, so we have the implications from (2) to (3) and from (3) to (4). A proof that (4) implies (1) can be found in \cite{chiba_96}. Therefore, assertions (1)–(4) are all equivalent.
\end{proof}
The next result is an important intermediate theorem.
\begin{proposition}\label{MdGPI}
	Let $D$ be a division ring with infinite center $F$, and $\sigma$ be an anti-automorphism of $D$. If $D$ satisfies a $\sigma^m$-GPI, then $D$ satisfies a GPI. Furthermore, if $D$ satisfies a $\sigma^m$-PI, then $D$ satisfies a PI.
\end{proposition}	

\begin{proof} According to Corollary \ref{2.3}, if $D$ satisfies a $\sigma^m$-GPI, then we can assume that $D$ satisfies a $\sigma^m$-linear polynomial identity
	$$
	f(x_1, x_2, \dots, x_n, x_1^{\sigma}, x_2^{\sigma} ,\dots, x_n^{\sigma}, \dots,x_1^{\sigma^m}, x_2^{\sigma^m},\dots, x_n^{\sigma^m} )\in D\left\langle X_m \right\rangle .
	$$
	Then, for all $r_1, r_2, \dots, r_n\in D$, we have
	$$
	f(r_1, r_2, \dots, r_n, r_1^{\sigma}, r_2^{\sigma} ,\dots, r_n^{\sigma}, \dots,r_1^{\sigma^m}, r_2^{\sigma^m}, \dots, r_n^{\sigma^m} )=0.
	$$
	Let 
	$$
	f_1(x_1, x_1^{\sigma}, \dots, x_1^{\sigma^m} )=f(x_1, r_2, \dots, r_n, x_1^{\sigma}, r_2^{\sigma} ,\dots, r_n^{\sigma}, \dots,x_1^{\sigma^m}, r_2^{\sigma^m}, \dots, r_n^{\sigma^m} ).
	$$
	If $f_1(x_1, x_1^{\sigma}, \dots, x_1^{\sigma^m})\neq 0$, then by using Lemma \ref{2.4}, $\dim_FD<\infty$. Therefore, according to Lemma~\ref{l3.1}, the division ring $D$ satisfies a GPI. Suppose $f_1(x_1, x_1^{\sigma}, \dots, x_1^{\sigma^m})=0$. For all $r_1, r_2, \dots, r_n, r_{11}, r_{12}, \dots ,r_{1m}\in D$, we have
	$$
	f(r_1, r_2, \dots, r_n, r_{11}, r_2^{\sigma} ,\dots, r_n^{\sigma}, \dots,r_{1m}, r_2^{\sigma^m}, \dots, r_n^{\sigma^m} )=0.
	$$ 
	Next, let 
	$$
	f_2(x_2, x_2^{\sigma},\dots , x_2^{\sigma^m})	= f(r_1, x_2,r_3 \dots, r_n, r_{11}, x_2^{\sigma},r_3^{\sigma}  ,\dots, r_n^{\sigma}, \dots,r_{1m}, x_2^{\sigma^m}, r_3^{\sigma^m}, \dots, r_n^{\sigma^m}).
	$$
	If $f_2(x_2, x_2^{\sigma},\dots , x_2^{\sigma^m})\ne0$, then $D$ satisfies a GPI. Conversely, if
	$$
	f_2(x_2, x_2^{\sigma},\dots , x_2^{\sigma^m})=0
	$$
	then for all  $r_1, r_2, \dots, r_n, r_{11}, r_{12}, \dots ,r_{1m}, r_{21}, r_{22}, \dots , r_{2m}\in D$, we have 
	$$
	f(r_1, r_2, \dots, r_n, r_{11}, r_{21} ,\dots, r_n^{\sigma}, \dots,r_{1m}, r_{2m}, \dots, r_n^{\sigma^m} )=0.
	$$ 
	By repeating the same method for the variables $x_3, x_4, \dots, x_n$, we can assume that for all $r_1, r_2, \dots , r_{nm}\in D$, 
	$$
	f(r_1, r_2, \dots, r_{11}, r_{21}, \dots, r_{n1}, \dots, r_{1m}, r_{2m}, \dots , r_{nm} )=0.  
	$$
	Thus, $D$ satisfies a GPI.
\end{proof}

To show the main result of this paper, we borrow the following well known property.
\begin{lemma}\label{l3.3} Let $F$ be a field and $R$ be a ring whose center contains $F$. If 
\[
f(t)=a_nt^n+a_{n-1}t^{n-1}+\dots+a_1t+a_0
\]
is a nonzero polynomial over $R$, then $f(t)=0$ has at most $n$ solutions in $F$. 
\end{lemma}
\begin{proof} This lemma is a corollary of \cite[Proposition 2.3.27]{15}.
\end{proof}
Now we are ready to prove the main theorem.

\begin{proof}[The proof of {\rm Theorem~\ref{main}}] Let $$f=f(x_1,\dots,x_n,x_1^\sigma,\dots,x_n^\sigma,\dots,x_1^{\sigma^m},\dots,x_n^{\sigma^m})$$ be a $\sigma^m$-GRI of $D$ and $f$ is defined at $$(r_1,\dots,r_n,r_1^\sigma,\dots,r_n^\sigma,\dots,r_1^{\sigma^m},\dots,r_n^{\sigma^m}).$$
	Assume that $t$ is a commuting indeterminate over $D$ and $t^\sigma=t$. Then, $$(r_j+x_jt)^{\sigma^i}=r_j^{\sigma^i}+x_j^{\sigma^i}t.$$ According to \cite[Lemma 7]{chiba}, $f$ is also defined at
	$$(r_1+x_1t,\dots,r_n+x_nt,r_1^\sigma+x_1^\sigma t,\dots,r_n^\sigma+x_n^\sigma t,\dots,r_1^{\sigma^m}+x_1^{\sigma^m}t,\dots,r_n^{\sigma^m}+x_n^{\sigma^m}t)$$ which is in $D(X_m)((t))^n$.
	Here $D(X_m)((t))^n$ is the ring of Laurent series over $D(X_m)$. Furthermore, we can write $f$ as
	\begin{align*}
& \ f(r_1+x_1t,\dots,r_n+x_nt,r_1^\sigma+x_1^\sigma t,\dots,r_n^\sigma+x_n^\sigma t,\dots,r_1^{\sigma^m}+x_1^{\sigma^m}t,\dots,r_n^{\sigma^m}+x_n^{\sigma^m}t)\\
	=& \ f_0+\sum_{i=1}^\infty f_i(x_1,\dots,x_n,x_1^\sigma,\dots,x_n^\sigma,\dots,x_1^{\sigma^m},\dots,x_n^{\sigma^m})t^i
	\end{align*}
	where $f_0=f(r_1,\dots,r_n,r_1^\sigma,\dots,r_n^\sigma,\dots,r_1^{\sigma^m},\dots,r_n^{\sigma^m})$ and $f_i$ $(i\geq 1)$ are homogeneous polynomials of degree $i$ in $D\langle X_m\rangle$. Moreover, there exists $i_0$ such that $f_{i_0}\neq 0$. We have
	\begin{align*}
& \ f(r_1+d_1t,\dots,r_n+d_nt,r_1^\sigma+d_1^\sigma t,\dots,r_n^\sigma+d_n^\sigma t,\dots,r_1^{\sigma^m}+d_1^{\sigma^m}t,\dots,r_n^{\sigma^m}+d_n^{\sigma^m}t)\\
	=& \ \sum_{i=1}^\infty f_i(d_1,\dots,d_n,d_1^\sigma,\dots,d_n^\sigma,\dots,d_1^{\sigma^m},\dots,d_n^{\sigma^m})t^i.
	\end{align*}
	Since $f(r_1+d_1t,\dots,r_n+d_nt,r_1^\sigma+d_1^\sigma t,\dots,r_n^\sigma+d_n^\sigma t,\dots,r_1^{\sigma^m}+d_1^{\sigma^m}t,\dots,r_n^{\sigma^m}+d_n^{\sigma^m}t)$ is an element in the quotient ring $D(t)$ of the polynomial ring $D[t]$, it can be represented as
	$$f(r_1+d_1t,\dots,r_n+d_nt,r_1^\sigma+d_1^\sigma t,\dots,r_n^\sigma+d_n^\sigma t,\\ \dots,r_1^{\sigma^m}+d_1^{\sigma^m}t,\dots,r_n^{\sigma^m}+d_n^{\sigma^m}t)=\frac{g(t)}{h(t)}$$
	where $g(t),h(t)\in D[t]$ and $h(t)\neq 0$. Since the field $F$ is infinite, there exist infinitely many $p\in F$ such that $f$ is defined at $(r_1+d_1p,\dots,r_n+d_np)$. Furthermore, since $f$ is a $\sigma^m$-GRI of $D$, we have
	\begin{center}
$f(r_1+d_1p,\dots,r_n+d_np,r_1^\sigma+d_1^\sigma p,\dots,r_n^\sigma+d_n^\sigma p,\dots,r_1^{\sigma^m}+d_1^{\sigma^m}p,\dots,r_n^{\sigma^m}+d_n^{\sigma^m}p)
	=  \frac{g(p)}{h(p)}=0,$
	\end{center}
	implying $g(p)=0$. Thus, the polynomial $g(t)$ has infinitely many roots in $F$, so $g(t)\equiv 0$ by Lemma~\ref{l3.3}. In this case, $$f(r_1+d_1t,\dots,r_n+d_nt,r_1^\sigma+d_1^\sigma t,\dots,r_n^\sigma+d_n^\sigma t,\dots,r_1^{\sigma^m}+d_1^{\sigma^m}t,\dots,r_n^{\sigma^m}+d_n^{\sigma^m}t)=0,$$ hence
	\begin{align*}
	&\sum_{i=1}^\infty f_i(d_1,\dots,d_n,d_1^\sigma,\dots,d_n^\sigma,\dots,d_1^{\sigma^m},\dots,d_n^{\sigma^m})t^i=0.
	\end{align*}
	Therefore, 
    \[
    f_i(d_1, d_2, \dots,d_n,d_1^{\sigma}, d_2^{\sigma},  \dots,d_n^{\sigma},\dots, d_1^{\sigma^m}, d_2^{\sigma^m}, \dots, d_n^{\sigma^m})=0
    \]
    for all $i$. Consequently, for all $d_1, d_2, \dots, d_n\in D$, we have $$f_{i_0}(d_1, d_2, \dots,d_n,d_1^{\sigma}, d_2^{\sigma},  \dots,d_n^{\sigma},\dots, d_1^{\sigma^m}, d_2^{\sigma^m}, \dots, d_n^{\sigma^m})=0.$$ Thus, $$f_{i_0}(x_1, x_2, \dots, x_n, x_1^{\sigma}, x_2^{\sigma} ,\dots, x_n^{\sigma}, \dots,x_1^{\sigma^m}, x_2^{\sigma^m}, \dots, x_n^{\sigma^m} )$$ is a $\sigma^m$-$GPI$ of $D$. According to Proposition \ref{MdGPI}, the ring $D$ satisfies a GPI. Therefore, $D$ is centrally finite by Lemma~\ref{l3.1}.
\end{proof}
% A table of contents will be automatically inserted in your article if it
% has 3 or more sections.  Please, do not try to manually change this
% behaviour.

% Also, please consider the following suggestions while preparing your 
% manuscript (as they will speed up the editorial process):
% * Avoid starting a new sentence with a mathematical formula;
% * Try to separate adjacent formulas with words;
% * Avoid inline formulas longer than half of a line. You can use math 
%   displays (\[...\]) instead;
% * Consider the use the enumerate and itemize environments for lists;
% * Consider the use of \dots, \ldots, \dotsc, \cdot, etc, instead of "..." 
%   or ".";
% * Instead of numbering or citing an article by hand (using parenthesis or 
%   brackets), consider the use of \cite, \ref and \eqref for citations and
%   cross-references;
% * Try to avoid inserting horizontal or vertical spacing, such as \hskip, 
%   \vskip and \bigskip;
% * Try to avoid inserting line or page brakes, such as \\, \newpage and
%   \clearpage.

% Acknowledgments should be added at the end of this section (right before
% the refences section) as a \subsection* (a subsection without a number):
% \subsection*{Acknowledgments} ...

%%% REFERENCES %%%
%{\small\bibliography{cimart}}
	\bibliographystyle{amsplain}

% Please, do not change the above line and do not insert your references
% into this file.  Instead, insert your references into the cimart.bib file.
% See cimart.bib for further instructions.

\EditInfo{March 27, 2025}{October 4, 2025}{Roozbeh Hazrat}

\end{document}